\documentclass[12pt]{article}

\usepackage{amsfonts,amssymb,amsmath,amsthm,epsfig,euscript}

\newtheorem{theorem}{Theorem}
\newtheorem{lemma}[theorem]{Lemma}
\newtheorem{corollary}{Corollary}

\theoremstyle{definition}
{

}

\long\def\symbolfootnote[#1]#2{\begingroup
\def\thefootnote{\fnsymbol{footnote}}\footnote[#1]{#2}\endgroup}

\newcommand{\red}[1][\sigma]{\mathrm{red}(#1)}

\newcommand{\sg}{\sigma}

\newcommand{\bx}[1][\sigma]{\mathrm{1\mbox{-}box}(#1)}

\newcommand{\fig}[2]{\begin{figure}[ht]
\centerline{\scalebox{.66}{\epsfig{file=#1.eps}}}
\caption{#2}
\label{fig:#1}
\end{figure}}

\title{The 1-box pattern on pattern avoiding permutations}

\author{
Sergey Kitaev \\
\small Department of Computer and Information Sciences\\[-0.8ex]
\small University of Strathclyde\\[-0.8ex]
\small Glasgow G1 1XH, United Kingdom\\[-0.8ex]
\small \texttt{sergey.kitaev@cis.strath.ac.uk}
\and
Jeffrey Remmel \\
\small Department of Mathematics\\[-0.8ex]
\small University of California, San Diego\\[-0.8ex]
\small La Jolla, CA 92093-0112. USA\\[-0.8ex]
\small \texttt{jremmel@ucsd.edu}
}

\date{\small Submitted: Date 1;  Accepted: Date 2;
 Published: Date 3.\\
\small MR Subject Classifications: 05A15, 05E05}

\begin{document}
\maketitle

\begin{abstract}
\noindent \ This paper is continuation of the study of the 1-box pattern 
in permutations introduced by the authors in \cite{kitrem4}. We derive a two-variable generating function for the distribution of this pattern on 132-avoiding permutations, and then study some of its coefficients providing a link to the Fibonacci numbers. We also find the number of separable permutations with two and three occurrences of the 1-box pattern. \\

\end{abstract}

\section{Introduction}

In this paper, we study {\em $1$-box patterns}, a particular case of {\em 
$(a,b)$-rectangular patterns} introduced in \cite{kitrem4}. 
That is, let $\sigma = \sg_1 \cdots \sg_n$ be a permutation written in one-line notation. Then we will consider the 
graph of $\sg$, $G(\sg)$, to be the set of points $(i,\sg_i)$ for 
$i =1, \ldots, n$.  For example, the graph of the permutation 
$\sg = 471569283$ is pictured in Figure 
\ref{fig:basic}.  

\fig{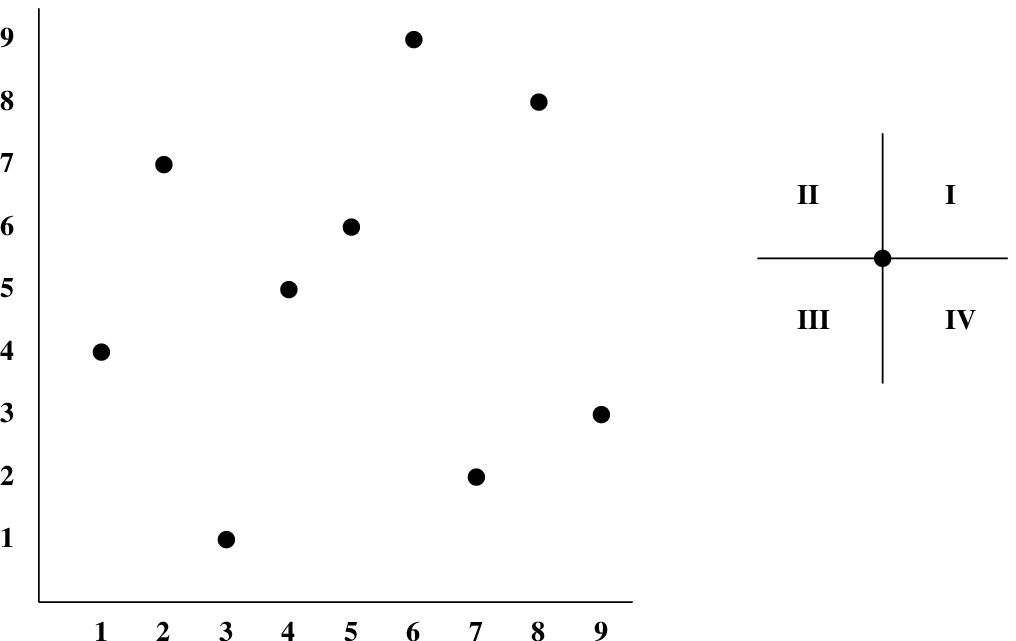}{The graph of $\sg = 471569283$.}

Then if we draw a coordinate system centered at a 
point $(i,\sg_i)$, we will be interested in  the points that 
lie in the $2a \times 2b$ rectangle centered at the origin. That is, 
the $(a,b)$-rectangle pattern 
 centered at $(i,\sg_i)$ equals the set of points  
$(i\pm r,\sg_i \pm s)$ such that $r\in \{0, \ldots, a\}$ and 
$s \in \{0, \ldots, b\}$.  Thus $\sg_i$ matches the $(a,b)$-rectangle  pattern in $\sg$, if there is 
at least one point in the $2a \times 2b $-rectangle centered at the point 
$(i,\sg_i)$ in $G(\sg)$ other than $(i,\sg_i)$. For example, when we
look for matches of the (2,3)-rectangle patterns, we would 
look at $4 \times 6$ rectangles centered at the point $(i,\sg_i)$ as 
pictured in Figure \ref{fig:basic4}.

\fig{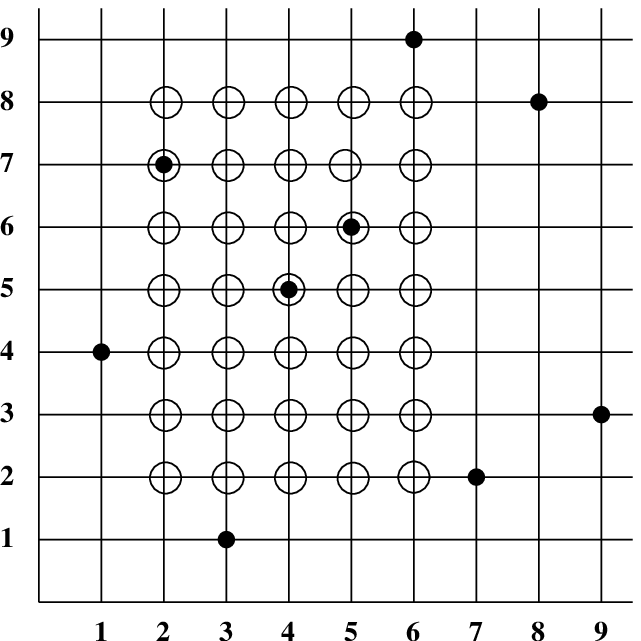}{The $4 \times 8$-rectangle centered at the point 
$(4,5)$ in the graph of $\sg = 471569283$.}

We shall refer to the $(k,k)$-rectangle pattern as 
the $k$-box pattern. 
For example, if $\sg = 471569283$, then the 2-box centered 
at the point $(4,5)$ in $G(\sg)$ is the set of circled 
points pictured in Figure 
\ref{fig:basic3}.  Hence, 
$\sg_i$ matches the $k$-box pattern in $\sg$, if there is 
at least one point in the $k$-box centered at the point 
$(i,\sg_i)$ in $G(\sg)$ other than $(i,\sg_i)$. 
For example, $\sg_4$ matches 
the pattern $k$-box for all $k \geq 1$ in $\sg = 471569283$ since 
the point $(5,6)$ is present in the $k$-box centered at the point $(4,5)$ in 
$G(\sg)$ for all $k \geq 1$. 
However, $\sg_3$ only matches the $k$-box pattern 
in  $\sg = 471569283$ for $k \geq 3$ since there are no points 
in 1-box or 2-box centered at $(3,1)$ in $G(\sg)$, but 
the point $(1,4)$ is in the 3-box centered at $(3,1)$ in $G(\sg)$. 
For $k \geq 1$, we let $k\mbox{-box}(\sg)$ denote the set of 
all $i$ such that $\sg_i$ matches the $k$-box pattern in 
$\sg = \sg_1 \cdots \sg_n$.

\fig{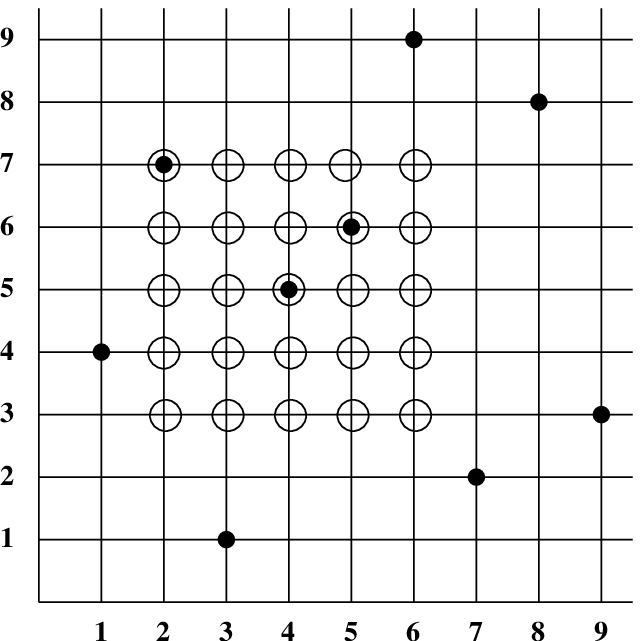}{The 2-box centered at the point 
$(4,5)$ in the graph of $\sg = 471569283$.}

Note that $\sg_i$ matches the 1-box pattern in $\sg$ if 
either $|\sg_i - \sg_{i+1}|=1$ or $|\sg_{i-1}-\sg_i|=1$. 
For example, the distribution of 
$\bx[\sg]$ for $S_2$, $S_3$, and $S_4$ is given below, where $S_n$ is the set of all permutations of length $n$. 

\begin{center}
\begin{tabular}{ccc}

\begin{tabular}{|c|c|}
\hline
$\sg$ &  $\bx[\sg]$ \\
\hline
12 & 2\\
\hline
21 & 2 \\
\hline
\end{tabular}

& \ \ \ \ \ &

\begin{tabular}{|c|c|}
\hline
$\sg$ &  $\bx[\sg]$ \\
\hline
123 & 3 \\
\hline 
132 & 2 \\
\hline
213 & 2 \\
\hline
231 & 2 \\
\hline
312 & 2  \\
\hline
321 & 3  \\
\hline
\end{tabular}\\

\end{tabular}
\end{center}

\begin{center}
\begin{tabular}{|c|c|c|c|c|}
\hline
$\sg$ &  $\bx[\sg]$ & \ & $\sg$ & $\bx[\sg]$ \\
\hline
1234 & 4 &  \ & 2134 & 4 \\
\hline 
1243 & 4 &  \ & 2143 & 4 \\
\hline
1324 & 2 &  \ & 2314 & 2 \\
\hline
1342 & 2 &  \ & 2341 & 3 \\
\hline
1423 & 2   & \ & 2413 & 0 \\
\hline
1432 & 3   & \ & \ 2431 & 2 \\
\hline
3124 & 2  &\ & 4123 & 3 \\
\hline 
3142 & 0 &\ & 4132 & 2 \\
\hline
3214 & 3 &  \ & 4213 & 2 \\
\hline
3241 & 2 & \ & 4231 & 2 \\
\hline
3412 & 4 & \ & 4312 & 4 \\
\hline
3421 & 4  & \ &  4321 & 4 \\
\hline
\end{tabular}
\end{center}

The notion of $k$-box patterns is related to the {\em mesh patterns} introduced by Br\"and\'en and Claesson \cite{BrCl} to provide explicit expansions for certain permutation statistics as, possibly infinite, linear combinations of (classical) permutation patterns.  This notion was further studied in \cite{AKV,kitlie,kitrem,kitremtie,kitremtieII,kitremtieIII,Ulf}. In particular, Kitaev and Remmel \cite{kitrem} initiated the systematic study of distribution of marked mesh patterns on permutations, and this study was extended to 132-avoiding permutations by Kitaev, Remmel and Tiefenbruck in \cite{kitremtie,kitremtieII,kitremtieIII}.

In this paper, we shall study the distribution of the 1-box pattern 
in 132-avoiding permutations and separable permutations.
Given a sequence $\sg = \sg_1 \cdots \sg_n$ of distinct integers,
let $\red[\sg]$ be the permutation found by replacing the
$i$-th largest integer that appears in $\sg$ by $i$.  For
example, if $\sg = 2754$, then $\red[\sg] = 1432$.  Given a
permutation $\tau=\tau_1 \cdots \tau_j$ in the symmetric group $S_j$, we say that the pattern $\tau$ {\em occurs} in $\sg = \sg_1 \ldots \sg_n \in S_n$ provided   there exists 
$1 \leq i_1 < \cdots < i_j \leq n$ such that 
$\red[\sg_{i_1} \cdots \sg_{i_j}] = \tau$.   We say 
that a permutation $\sg$ {\em avoids} the pattern $\tau$ if $\tau$ does not 
occur in $\sg$. In particular, a permutation $\sigma$ avoids the pattern 132 if $\sigma$ does not contain a subsequence of three elements, where the first element is the smallest one, and the second element is the largest one.
Let $S_n(\tau)$ denote the set of permutations in $S_n$ 
which avoid $\tau$. In the theory of permutation patterns (see \cite{kit} for a comprehensive introduction to the area),  $\tau$ is called a {\em classical pattern}. The results in this paper can be viewed 
as another contribution to the long line of research in the literature 
which studies various distributions on pattern-avoiding permutations (e.g. see \cite[Chapter 6.1.5]{kit} for relevant results).

The outline of this paper is as follows. 
In Section \ref{sec2} we  shall study 
the distribution of the 1-box pattern in 132-avoiding permutations. In particular, we shall derive explicit formulas for  the generating functions 
\begin{equation*}
A(t,x) = \sum_{n \geq 0} A_n(x)t^n,
\end{equation*}
\begin{equation*}
B(t,x) = \sum_{n \geq 1} B_n(x)t^n\ \mbox{and}
\end{equation*}
\begin{equation*}
E(t,x) = \sum_{n \geq 1} E_n(x)t^n 
\end{equation*}
where $A_0(x) = 1$ and for $n \geq 1$,  
\begin{eqnarray*}
A_n(x) &=& \sum_{\sg \in S_n(132)} x^{\bx[\sg]}\\
B_n(x) &=& \sum_{\sg = \sg_1 \ldots \sg_n \in S_n(132),\sg_1 =n} 
x^{\bx[\sg]}\ \mbox{and} \\
E_n(x) &=& \sum_{\sg = \sg_1 \ldots \sg_n \in S_n(132),\sg_n =n} 
x^{\bx[\sg]}.
\end{eqnarray*}
In Section \ref{sec3}, we shall study the coefficients 
of $x^k$ in the polynomials $A_n(x)$, $B_n(x)$, and 
$E_n(x)$ for $k \in \{0,1,2,3,4\}$ as well as the coefficient 
of the highest power of $x$ in these polynomials. Many of 
these coefficients can be expressed in terms of the 
Fibonacci numbers $F_n$.  For example, for $n \geq 2$, 
the coefficient of $x^2$ in $A_n(x)$ is $F_n$ and the 
coefficient of $x^2$ in $B_n(x)$ and $E_n(x)$ is $F_{n-2}$. 
Finally, in Section \ref{sec4}, we shall study the 1-box pattern on {\em separable permutations}.

\section{Distribution of the 1-box pattern on 132-avoiding permutations}\label{sec2}

In this section, we shall study the generating functions 
$A(t,x)$, $B(t,x)$, and $E(t,x)$.  Clearly, 
$A_1(x) = B_1(x) = E_1(x) =1$. 
One can see from our tables for $S_2$, $S_3$, and $S_4$ that  
$A_2(x) = 2x^2$, $A_3(x) = 3x^2 + 2x^3$, and $A_4(x)=5x^2+3x^3+6x^4$. 
Similarly, one can check that 
$B_2(x) = E_2(x) = x^2$, $B_3(x) = E_3(x) = x^2+x^3$, and 
 $B_4(x)=E_4(x)=2x^2+x^3+2x^4$.

We shall classify the $132$-avoiding permutations 
$\sg = \sg_1 \cdots \sg_n$ by position of $n$ 
in $\sg$. That is, let 
$S^{(i)}_n(132)$ denote the set of $\sg \in S_n(132)$ such 
that $\sg_i =n$. Clearly each $\sg \in  S_n^{(i)}(132)$ has the structure 
pictured in Figure \ref{fig:basic2}. That is, in the graph of 
$\sg$, the elements to the left of $n$, $A_i(\sg)$, have 
the structure of a $132$-avoiding permutation, the elements 
to the right of $n$, $B_i(\sg)$, have the structure of a 
$132$-avoiding permutation, and all the elements in 
$A_i(\sg)$ lie above all the elements in 
$B_i(\sg)$.  Note that the number of $132$-avoiding 
permutations in $S_n$ is the Catalan number 
$C_n = \frac{1}{n+1} \binom{2n}{n}$, which is a well-known fact, and the generating 
function for the $C_n$'s is given by 
$$C(t) = \sum_{n \geq 0} C_n t^n = \frac{1-\sqrt{1-4t}}{2t}=
\frac{2}{1+\sqrt{1-4t}}.$$

\fig{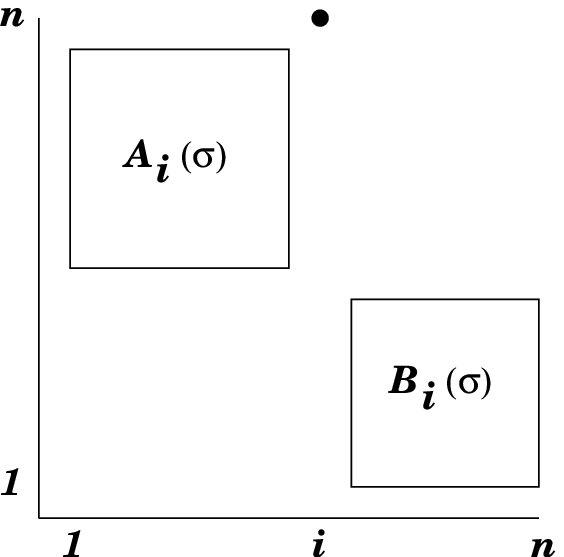}{The structure of $132$-avoiding permutations.}


The following lemma establishes relations between $A_n(x)$, $B_n(x)$, and $E_n(x)$.

\begin{lemma}\label{lemma1} For all $n \geq 1$, 
$B_n(x)=E_n(x)$ and for $n\geq 4$,
\begin{equation}\label{BArecur}
B_n(x)=x^n+(A_{n-1}(x)-B_{n-1}(x))+\sum_{i=2}^{n-2}x^{n-i}(A_i(x)-B_i(x)).
\end{equation}
For $n \geq 2$,  
\begin{equation}\label{ABrecur}
A_n(x)=B_n(x)+\sum_{i=2}^{n}B_i(x)A_{n-i}(x).
\end{equation}
\end{lemma}

\begin{proof} We begin with deriving relationships for $B_n(x)$ and $E_n(x)$. Any 132-avoiding permutation $\pi=\pi_1\cdots\pi_n$ beginning with the largest letter $n$ is of one of the three forms described below:
\begin{enumerate} 
\item the decreasing permutation $n(n-1)\cdots 1$;
\item  $n\ell\pi_3\pi_4\cdots\pi_n$ where $\ell<n-1$ and $\ell\pi_{3}\pi_{4}\cdots\pi_n$ is a 132-avoiding permutation on $\{1,\ldots,n-1\}$;
\item $n(n-1)\cdots (n-i+1)\ell\pi_{i+2}\pi_{i+3}\cdots\pi_n$, where $2\leq i\leq n-2$, $\ell<n-i$ and $\ell\pi_{i+2}\pi_{i+3}\cdots\pi_n$ is a 132-avoiding permutation on $\{1,\ldots,n-i\}$.
\end{enumerate}

This structural observation implies immediately (\ref{BArecur}). Indeed, in the decreasing permutation each element is an occurrence of the 1-box pattern thus giving a contribution of $x^n$ to the function $B_n(x)$. Also, in the second case, $n$ is not an occurrence of the 1-box pattern in $\pi$ 
and it does not effect 
whether any of the remaining elements in $\pi$ are occurrences of 
the 1-box pattern in $\pi$. Thus, in this case we have a contribution of $(A_{n-1}(x)-B_{n-1}(x))$ to $B_n(x)$. Finally, in the last case, for any $i$, $2\leq i\leq n-2$, each of the elements $n-i+1, n-i+2,\ldots, n$ is an occurrence of the 1-box pattern in $\pi$ and these elements do not effect 
whether any of the remaining elements in $\pi$ are occurrences of 
the 1-box pattern in $\pi$. Thus, in this case we have a contribution of $\sum_{i=2}^{n-2}x^{n-i}(A_i(x)-B_i(x))$ to $B_n(x)$.

We can use similar methods to prove that for all $n \geq 4$, 
\begin{equation}\label{EArecur}
E_n(x)=x^n+(A_{n-1}(x)-E_{n-1}(x))+\sum_{i=2}^{n-2}x^{n-i}(A_i(x)-E_i(x)).
\end{equation}
That is, if $\pi$ is a 132-avoiding permutation in $S_n$ that 
ends in $n$, we have the following three cases:
\begin{enumerate} 
\item $\pi$ is the increasing permutation $1\cdots n$;
\item  $\pi=\pi_1\cdots\pi_{n-2}\ell n$ where $\ell<n-1$ and $\pi_{1}\cdots\pi_{n-2}\ell$ is a 132-avoiding permutation on $\{1,\ldots,n-1\}$;
\item $\pi_{1}\cdots\pi_{n-i-1}\ell(n-i+1)(n-i+2)\cdots n$, where $2\leq i\leq n-2$, $\ell<n-i$ and $\pi_{1}\cdots\pi_{n-i-1}\ell$ is a 132-avoiding permutation on $\{1,\ldots,n-i\}$.
\end{enumerate}

Arguing as above, we see that the identity permutations 
contributes $x^n$ to $E_n(x)$, the elements in case (2) contribute 
$A_{n-1}(x) - E_{n-1}(x)$ to $E_n(x)$, and the elements in 
case (3) contribute $\sum_{i=2}^{n-2}x^{n-i}(A_i(x)-E_i(x))$ to 
$E_n(x)$.

Given that we have computed that  
$B_n(x) =E_n(x)$ for $1 \leq n \leq 3$, one can easily use 
(\ref{BArecur}) and (\ref{EArecur}) to prove that  
$B_n(x) = E_n(x)$ for all $n \geq 1$ by induction.

To prove  (\ref{ABrecur}), note that $S_n(132)=S^{(1)}_n(132)\cup S^{(n)}_n(132)\cup_{2\leq i\leq n-1}S^{(i)}_n(132)$. Clearly, 
the permutations in $S^{(1)}_n(132)$ contribute $B_n(x)$ to 
$A_n(x)$ and the permutations in $S^{(n)}_n(132)$ contribute $E_n(x)$ to 
$A_n(x)$. Now suppose that 
$2 \leq i \leq n$ and $\pi = \pi_1 \cdots \pi_n 
\in S^{(i)}_n(132)$.  Then all the elements in 
$\pi_1 \cdots \pi_{i-1}$ are strictly greater than all the elements 
in $\pi_{i+1} \cdots \pi_{n}$. It follows that 
$\pi_{i+1}\leq n-2$. Hence the elements 
$\pi_1 \cdots \pi_{i-1}n$ have no effect as to whether 
any of the elements in $\pi_{i+1} \cdots \pi_{n}$ are 
occurrences of the 1-box pattern in $\pi$. Hence 
the elements $S^{(i)}_n(132)$ contribute $E_i(x) A_{n-i}(x)$ to 
$A_n(x)$. Thus for all $n \geq 2$, 
\begin{equation}\label{ABErec}
A_n(x)=B_n(x)+E_n(x)+\sum_{i=2}^{n}E_i(x)A_{n-i}(x).
\end{equation}
It is easy to see that 
since  $B_n(x)=E_n(x)$ for all $n\geq 1$, (\ref{ABErec}) implies 
(\ref{ABrecur}). 
\end{proof}

The following theorem gives the generating function for the entire distribution of the 1-box pattern over 132-avoiding permutations.

\begin{theorem}\label{theorem-enum1}  We have
\begin{equation}\label{Atxform}
A(t,x)=
\frac{1+t+t^2-t x-t^2 x-t^3 x+t^3 x^2-\sqrt{F(t,x)}}{2(t(1-xt)+x^2t^2)}
\end{equation}
where $F(t,x)=(1+t+t^2-t x-t^2 x-t^3 x+t^3 x^2)^2+4((1+t)(1-xt)+x^2t^2)(t(1-xt)+x^2t^2)$.
Also,
$$B(t,x)=E(t,x)=\frac{t(1-xt)+x^2t^2}{(1+t)(1-xt)+x^2t^2}A(t,x).$$
\end{theorem}

\begin{proof} Multiplying both parts of (\ref{ABrecur}) by $t^n$ and summing over all $n\geq 2$ we obtain
$$A(t,x)-(1+t)=(B(t,x)-t)+(B(t,x)-t)A(t,x).$$
Solving for $A(t,x)$, we obtain that 
\begin{equation}\label{AversusB}
A(t,x)=\frac{1+B(t,x)}{1+t-B(t,x)}.
\end{equation}
Now multiplying both parts of  (\ref{BArecur}) by $t^n$ and summing over all $n\geq 2$ we obtain
$$B(t,x)-(t+x^2t^2+(x^2+x^3)t^3)=\frac{x^4t^4}{1-xt}+t(A(t,x)-(1+t+2x^2t^2))$$
$$-t(B(t,x)-(t+x^2t^2))+\frac{x^2t^2}{1-xt}\left((A(t,x)-(1+t))-(B(t,x)-t)\right).$$
Solving for $B(t,x)$, we obtain that 
\begin{equation}\label{BversusA}
B(t,x)=\frac{t(1-xt)+x^2t^2}{(1+t)(1-xt)+x^2t^2}A(t,x).
\end{equation}

Combining (\ref{AversusB}) and (\ref{BversusA}),
we see that $A(t,x)$ satisfies the following quadratic equation  
$$(t(1-xt)+x^2t^2)A^2(t,x)-(1+t+t^2-tx-t^2x-t^3x+t^3x^2)A(t,x)+(1+t)(1-xt)+x^2t^2=0$$
which can be solved to yield (\ref{Atxform}). 
\end{proof}

We used Mathematica to find the first few terms of $A(t,x)$ and 
$B(t,x)=E(t,x)$. That is, we have that 
\begin{small}
\begin{eqnarray*} 
A(t,x) &=& 1+t+2 x^2t^2+x^2 (3+2 x)t^3+x^2 \left(5+3 x+6 x^2\right)t^4+x^2 \left(8+5 x+19 x^2+10 x^3\right)t^5+\\
&& x^2 \left(13+8 x+50 x^2+35 x^3+26 x^4\right)t^6+x^2
\left(21+13 x+119 x^2+95 x^3+127 x^4+54 x^5\right)t^7+\\
&& x^2 \left(34+21 x+265 x^2+230 x^3+451 x^4+295 x^5+134 x^6\right)t^8+\\
&& x^2 \left(55+34 x+564 x^2+517
x^3+1373 x^4+1118 x^5+895 x^6+306 x^7\right)t^9+\\
&& x^2 \left(89+55 x+1160 x^2+1107 x^3+3790 x^4+3548 x^5+4010 x^6+2283 x^7+754 x^8\right)t^{10}+\cdots.
\end{eqnarray*}
and  
\begin{eqnarray*}
B(t,x) &=& E(t,x)  \\
&=& t+x^2t^2+x^2 (1+x)t^3+x^2 \left(2+x+2 x^2\right)t^4+\\
&& x^2 \left(3+2 x+6 x^2+3 x^3\right)t^5 +x^2 \left(5+3 x+16 x^2+11 x^3+7 x^4\right)t^6+\\
&& x^2
\left(8+5 x+39 x^2+30 x^3+36 x^4+14 x^5\right)t^7+\\
&& x^2 \left(13+8 x+88 x^2+75 x^3+131 x^4+81 x^5+33 x^6\right)t^8+\\
&& x^2 \left(21+13 x+190 x^2+171 x^3+410
x^4+319 x^5+233 x^6+73 x^7\right)t^9+\\
&& x^2 \left(34+21 x+395 x^2+372 x^3+1156 x^4+1044 x^5+1087 x^6+579 x^7+174 x^8\right)t^{10}+\cdots.
\end{eqnarray*}
\end{small}

\section{Properties of coefficients of $A_n(x)$ and $B_n(x)=E_n(x)$}\label{sec3}In this section, we shall explain several of the coefficients of 
the polynomials $A_n(x)$ and $B_n(x)=E_n(x)$ and show their connections 
with the Fibonacci numbers.

In Subsection \ref{FibCo}, we study the coefficients of 
$x^k$ in the the polynomials $A_n(x)$ and $B_n(x)=E_n(x)$ for 
$k \in \{0,1,2,3,4\}$  and, in Subsection \ref{coHighest}, we derive the generating functions for the highest coefficients for these polynomials. 

\subsection{The four smallest coefficients and the Fibonacci numbers}\label{FibCo}

Clearly the coefficient of $x$ in either $A_n(x)$, $B_n(x)$, or $E_n(x)$ 
is 0 by the definition of an occurrence of the 1-box pattern. The following theorem states that for $n\geq 2$, each 132-avoiding permutation of length $n$ has at least two occurrences of the 1-box pattern. In what follows, we need the notion of the celebrated $n$-th {\em Fibonacci number} $F_n$ defined as $F_0=F_1=1$ and, for $n\geq 2$, $F_n=F_{n-1}+F_{n-2}$. Also, for a polynomial $P(x)$, we let $P(x)|_{x^m}$ denote the coefficient of $x^m$.  

\begin{theorem}\label{smallest} For $n \geq 2$, 
$A_n(x)|_{x^0} = B_n(x)|_{x^0} =E_n(x)|_{x^0} =0$. 
\end{theorem}

\begin{proof} Clearly, it is enough to prove the claim for $A_n(x)$. 
We proceed by induction on $n$.  The claim is clearly true for 
$n =2$.  Next suppose that $n\geq 3$ and $\sg = S_n(132)$. 
From the structure of 132-avoiding permutations presented in Figure  \ref{fig:basic2}, either $A_i(\sg)$ is empty in which case $B_i(\sg)$ has at least two elements and it contains an occurrence of the 1-box pattern by the induction hypothesis, or $A_i(\sg)$ has a single element $n-1$ leading to two occurrence of the pattern formed by $n$ and $n-1$, or $A_i(\sg)$ has at least two elements and we apply the induction hypothesis to it. \end{proof}

\begin{theorem} \label{fib1}
For $n\geq 2$, $A_n(x)|_{x^2}=F_n$ and $B_n(x)|_{x^2}=E_n(x)|_{x^2}=F_{n-2}$.
\end{theorem}

\begin{proof} 
We proceed by induction on $n$. Note that 
$A_2(x)|_{x^2}=2 =F_2$ and $B_2(x)|_{x^2}=E_2(x)|_{x^2}= 1 =F_0$. 
Similarly, 
$A_3(x)|_{x^2}=3 =F_3$ and $B_3(x)|_{x^2}=E_3(x)|_{x^2}= 1 =F_1$. 
Thus our claim holds for $n =2$ and $n=3$. 

For $n \geq 4$, it follows from (\ref{BArecur}) and Theorem \ref{smallest} 
that 
\begin{eqnarray*}
B_n(x)|_{x^2}&=& x^n|_{x^2}+(A_{n-1}(x)|_{x^2}-B_{n-1}(x)|_{x^2})+
\sum_{i=2}^{n-2} (x^{n-i}(A_i(x)-B_i(x)))|_{x^2}\\
&=& A_{n-1}(x)|_{x^2}-B_{n-1}(x)|_{x^2}  + (A_{n-2}(x)-B_{n-2}(x))|_{x^0} \\
&=& F_{n-1}-F_{n-3} = F_{n-2}.
\end{eqnarray*}
But then by (\ref{ABrecur}), we have that 
\begin{equation} \label{x2Arec} 
A_{n}(x)|_{x^2}=B_{n}(x)|_{x^2}+\sum_{i=2}^{n}(B_i(x)A_{n-i}(x))|_{x^2}.
\end{equation}
Note that since $n \geq 4$ and $2 \leq i \leq n$
\begin{eqnarray*}
(B_i(x)A_{n-i}(x))|_{x^2} &=& 
(B_i(x)|_{x^0}) (A_{n-i}(x))|_{x^2}) +  
(B_i(x)|_{x^1}) (A_{n-i}(x))|_{x^1}) + \\
&& (B_i(x)|_{x^2}) (A_{n-i}(x))|_{x^0}) \\
&=&(B_i(x)|_{x^2}) (A_{n-i}(x))|_{x^0})
\end{eqnarray*}
since $B_i(x)|_{x^1}=  A_{n-i}(x)|_{x^1} =0$ for 
$i \geq 1$ and  $B_i(x)|_{x^0} =0$ for $i \geq 2$. 
But then since $A_i(x)|_{x^0} =0$ for $i \geq 2$ and 
$A_i(x)|_{x^0} =1$ for $i= 0,1$, it follows that (\ref{x2Arec}) 
reduces to 
\begin{eqnarray*}
A_{n}(x)|_{x^2} &=& B_{n}(x)|_{x^2}+B_{n}(x)|_{x^2}+B_{n-1}(x)|_{x^2} \\
&=& F_{n-2} + F_{n-2} +F_{n-3} = F_{n-2} +F_{n-1} = F_n.
\end{eqnarray*} 
\end{proof}

\begin{corollary} For $n\geq 2$, the number of $132$-avoiding permutations of length $n$ that do not begin (resp. end) with $n$ and contain exactly two occurrences of the $1$-box pattern is $F_{n-1}$.\end{corollary}

\begin{proof} A proof is straightforward from Theorem \ref{fib1}, since $$A_n(x)|_{x^2}-B_n(x)|_{x^2}=A_n(x)|_{x^2}-E_n(x)|_{x^2}=F_{n-1}.$$\end{proof}

\begin{theorem} \label{fib2}
For $n\geq 3$, $A_n(x)|_{x^3}=F_{n-1}$ and $B_n(x)|_{x^3}=E_n(x)|_{x^3}=F_{n-3}$.
\end{theorem}

\begin{proof} 

We proceed by induction on $n$, the length of permutations, and the formulas (\ref{BArecur}) and (\ref{ABrecur}). Note that we have computed that 
$A_3(x)|_{x^2}=2 =F_2$,  $A_4(x)|_{x^2}=3 =F_3$,
$B_3(x)|_{x^2}=E_3(x)|_{x^2}= 1 =F_0$, and 
$B_4(x)|_{x^2}=E_4(x)|_{x^2}= 1 =F_1$.
Thus our claim holds for $n=3$ and $n =4$. 

For $n \geq 5$, it follows from (\ref{BArecur}) and Theorem \ref{smallest} 
that 
\begin{eqnarray*}
B_n(x)|_{x^3}&=& x^n|_{x^3}+(A_{n-1}(x)|_{x^3}-B_{n-1}(x)|_{x^3})+
\sum_{i=2}^{n-2} (x^{n-i}(A_i(x)-B_i(x)))|_{x^3}\\
&=& A_{n-1}(x)|_{x^3}-B_{n-1}(x)|_{x^3}  + (A_{n-2}(x)-B_{n-2}(x))|_{x^1}  + 
(A_{n-3}(x)-B_{n-3}(x))|_{x^0}
\\
&=& F_{n-2}-F_{n-4} = F_{n-3}.
\end{eqnarray*}
But then by (\ref{ABrecur}), we have that 
\begin{equation} \label{x3Arec} 
A_{n}(x)|_{x^3}=B_{n}(x)|_{x^3}+\sum_{i=2}^{n}(B_i(x)A_{n-i}(x))|_{x^3}.
\end{equation}
Note that since $n \geq 5$ and $2 \leq i \leq n$,
\begin{eqnarray*}
(B_i(x)A_{n-i}(x))|_{x^3} &=& (B_i(x)|_{x^0}) + (A_{n-i}(x)|_{x^3})
(B_i(x)|_{x^1}) (A_{n-i}(x)|_{x^2}) +  \\
&&
(B_i(x)|_{x^2}) (A_{n-i}(x)|_{x^1}) + (B_i(x)|_{x^3}) (A_{n-i}(x)|_{x^0}) \\
&=& (B_i(x)|_{x^3}) (A_{n-i}(x)|_{x^0})
\end{eqnarray*}
since $B_i(x)|_{x^1}=  A_{n-i}(x)|_{x^1} =0$ for 
$i \geq 1$ and  $B_i(x)|_{x^0} =0$ for $i \geq 2$. 
But then since $A_i(x)|_{x^0} =0$ for $i \geq 2$ and 
$A_i(x)|_{x^0} =1$ for $i= 0,1$, it follows that (\ref{x3Arec}) 
reduces to 
\begin{eqnarray*}
A_{n}(x)|_{x^3} &=& B_{n}(x)|_{x^3}+B_{n}(x)|_{x^3}+B_{n-1}(x)|_{x^3} \\
&=& F_{n-3} + F_{n-3} +F_{n-4} = F_{n-3} +F_{n-2} = F_{n-1}.
\end{eqnarray*} 
\end{proof}

\begin{corollary} For $n\geq 3$, the number of $132$-avoiding permutations of length $n$ that do not begin (resp. end) with $n$ and contain exactly three occurrences of the $1$-box pattern is $F_{n-2}$.\end{corollary}

\begin{proof} A proof is straightforward from Theorem \ref{fib2}, since $$A_n(x)|_{x^3}-B_n(x)|_{x^3}=A_n(x)|_{x^3}-E_n(x)|_{x^3}=F_{n-2}.$$\end{proof}

Regarding the number of 132-avoiding permutations with exactly four occurrences of the 1-box pattern, we can derive the following recurrence relations involving the Fibonacci numbers. 

\begin{theorem} \label{fib3} We have that for $n\leq 3$, $A_n(x)|_{x^4}=B_n(x)|_{x^4}=E_n(x)|_{x^4}=0$, $B_4(x)|_{x^4}=2$, $B_5(x)|_{x^4}=6$, and for 
$n\geq 4$, 
\begin{equation}\label{eqnAx^4}
A_n(x)|_{x^4}=2B_n(x)|_{x^4}+B_{n-1}(x)|_{x^4}+\sum_{i=2}^{n-2}F_{i-2}F_{n-i};
\end{equation}
while for $n\geq 6$,
\begin{equation}\label{eqnBx^4}
B_n(x)|_{x^4}=B_{n-1}(x)|_{x^4}+B_{n-2}(x)|_{x^4}+F_{n-1}+\sum_{i=4}^{n-3}F_{i-2}F_{n-1-i}.
\end{equation}
\end{theorem}
\begin{proof}  
The initial conditions follow from the expansions of $A(t,x)$ and $B(t,x)$ given above.

By (\ref{ABrecur}), we have that 
\begin{equation} \label{x4Arec} 
A_{n}(x)|_{x^4}=B_{n}(x)|_{x^4}+\sum_{i=2}^{n}(B_i(x)A_{n-i}(x))|_{x^4}.
\end{equation}
Note that since $n \geq 4$ and $2 \leq i \leq n$,
\begin{eqnarray*}
(B_i(x)A_{n-i}(x))|_{x^4} &=& (B_i(x)|_{x^0}) (A_{n-i}(x)|_{x^4})+ 
(B_i(x)|_{x^1}) (A_{n-i}(x)|_{x^3}) +  \\
&&
(B_i(x)|_{x^2}) (A_{n-i}(x)|_{x^2}) + (B_i(x)|_{x^3}) (A_{n-i}(x)|_{x^1}) +\\
&& (B_i(x)|_{x^4}) (A_{n-i}(x)|_{x^0}) \\
&=& (B_i(x)|_{x^2}) (A_{n-i}(x)|_{x^2}) + (B_i(x)|_{x^4}) (A_{n-i}(x)|_{x^0})
\end{eqnarray*}
since $B_i(x)|_{x^1}=  A_{n-i}(x)|_{x^1} =0$ for 
$i \geq 1$ and  $B_i(x)|_{x^0} =0$ for $i \geq 2$. 
But then since $A_i(x)|_{x^0} =0$ for $i \geq 2$ and 
$A_i(x)|_{x^0} =1$ for $i= 0,1$,
(\ref{x4Arec})  reduces to  
$$A_n(x)|_{x^4}=B_n(x)|_{x^4}+B_n(x)|_{x^4}+B_{n-1}(x)|_{x^4}+\sum_{i=2}^{n-2}\left(B_i(x)|_{x^2}\right)\left(A_{n-i}(x)|_{x^2}\right).$$
Then we can apply Theorem~\ref{fib1} to obtain (\ref{eqnAx^4}).

Let $n\geq 6$. From (\ref{BArecur}), 
$$B_n(x)|_{x^4}=\left(A_{n-1}(x)|_{x^4}-B_{n-1}(x)|_{x^4}\right)+\left(A_{n-2}(x)|_{x^2}-B_{n-2}(x)|_{x^2}\right),$$
since only the term corresponding to $i=n-2$ from the sum contributes to $x^4$.  
Applying (\ref{eqnAx^4}) and Theorem~\ref{fib1}, we obtain
\begin{eqnarray*}
B_n(x)|_{x^4} &=& \left(2B_{n-1}(x)|_{x^4}+B_{n-2}(x)|_{x^4}+\sum_{i=2}^{n-3}F_{i-2}F_{n-1-i}\right)-B_{n-1}(x)|_{x^4}+F_{n-2}-F_{n-4}\\
&=&  B_{n-1}(x)|_{x^4} +B_{n-2}(x)|_{x^4}+F_{n-3}+F_{n-4}+F_{n-2}-F_{n-4}+\sum_{i=4}^{n-3}F_{i-2}F_{n-1-i} \\
&=& B_{n-1}(x)|_{x^4} +B_{n-2}(x)|_{x^4}+F_{n-1} +\sum_{i=4}^{n-3}F_{i-2}F_{n-1-i}. 
\end{eqnarray*}

Note that $B_5(x)|_{x^4} =6$, $B_4(x)|_{x^4} =2$, $B_3(x)|_{x^4} =0$, 
and $F_4 =5$ so that (\ref{eqnBx^4}) does not hold for $n=5$. \end{proof}

We can use Theorem \ref{fib3} to find the generating functions 
for $A_n(x)|_{x^4}$ and $B_n(x)|_{x^4}$. That is, 
let 
\begin{equation*}
\mathbb{A}_4(t) = \sum_{n \geq 4} (A_n(x)|_{x^4}) t^n
\end{equation*}
and 
\begin{equation*}
\mathbb{B}_4(t) = \sum_{n \geq 4} (B_n(x)|_{x^4}) t^n.
\end{equation*}
Then we have the following theorem. 
\begin{theorem}\label{ABcoefx4}
\begin{equation} \label{Acoefx4}
\mathbb{A}_4(t) = \frac{t^4(6+t-7t^2-t^3+3t^4+t^5)}{(1-t-t^2)^3}
\end{equation}
and 
\begin{equation} \label{Bcoefx4}
\mathbb{B}_4(t) = \frac{t^4(2-t^2+t^3+t^4)}{(1-t-t^2)^3}.
\end{equation}
\end{theorem}
\begin{proof}
 First observe that 
\begin{eqnarray*}
\sum_{n \geq 7} \left( \sum_{i=4}^{n-3} F_{i-2}F_{n-1-i}\right) t^n &=& 
t^3 \sum_{n \geq 7} \left( \sum_{j=2}^{n-5} F_{j}F_{n-3-j}\right) t^{n-3} \\
&=& t^3 \sum_{n \geq 4} \left( \sum_{j=2}^{n-2} F_{j}F_{n-j}\right) t^{n} \\
&=& t^3 \left(\sum_{j \geq 2} F_j t^j\right)^2.
\end{eqnarray*}
Using the fact that $\sum_{n \geq 0} F_n t^n = \frac{1}{1-t-t^2}$, it 
follows that 
\begin{eqnarray*}
\sum_{n \geq 7} \left( \sum_{i=4}^{n-3} F_{i-2}F_{n-1-i}\right) t^n &=& 
t^3\left( \frac{1}{1-t-t^2} -(1+t)\right)^2 \\
&=& 
t^3 \frac{(t^2(2+t))^2}{(1-t-t^2)^2} = \frac{(2+t)^2t^7}{(1-t-t^2)^2}. 
\end{eqnarray*}

Next observe that 
\begin{equation*} 
\sum_{n \geq 6} F_{n-1}t^n = t\left(\frac{1}{1-t-t^2} -(1+t+2t^2 
+3t^3+5t^4) \right) = \frac{(8+5t)t^5}{1-t-t^2}.
\end{equation*}
Thus 
\begin{eqnarray*} 
H(t) &=& \sum_{n \geq 6} H_n t^n \\
&=& \sum_{n \geq 6} F_{n-1}t^n + \sum_{n \geq 7} \left( \sum_{i=4}^{n-3} F_{i-2}F_{n-1-i}\right) t^n \\
&=& \frac{(2+t)^2t^7}{(1-t-t^2)^2} + \frac{(8+5t)t^5}{1-t-t^2}\\
&=& \frac{(8+t-9t^2-4t^3)t^6}{(1-t-t^2)^2}.
\end{eqnarray*}
Here we use Mathematica to simplify the last expression.

We can now rewrite (\ref{eqnBx^4}) as 
\begin{equation} \label{newB4rec}
B_n(x)|{x^4} = B_{n-1}(x)|{x^4} + B_{n-2}(x)|{x^4} +H_n
\end{equation}
for $n \geq 6$. 
Multiplying both sides of (\ref{newB4rec}) by $t^n$ and summing 
for $n \geq 6$, we see that 
$$\mathbb{B}_4(t) - 2t^4 -6t^5 = t(\mathbb{B}_4(t)-2t^4) + 
t^2 \mathbb{B}_4(t) + H(t).$$
Solving for $\mathbb{B}_4(t)$ and using Mathematica, we obtain 
that 
$$\mathbb{B}_4(t) = \frac{t^4(2-t^2+t^3+t^4)}{(1-t-t^2)^3}.$$

Next observe that 
\begin{eqnarray*} 
\sum_{n \geq 4} \left( \sum_{i=2}^{n-2} F_{i-2}F_{n-i}\right) t^n &=& 
\sum_{n \geq 4} \left( \sum_{j=0}^{n-4} F_{j}F_{n-2-j}\right) t^n \\
&=& t^2 \sum_{n \geq 4} \left( \sum_{j=0}^{n-4} F_{j}F_{n-2-j}\right) t^{n-2} \\&=& t^2 \left(\sum_{j \geq 0} F_j t^j\right)
\left(\sum_{j \geq 0} F_j t^j-(1+t)\right) \\
&=& \frac{(2+t)t^4}{(1-t-t^2)^2}. 
\end{eqnarray*}
Thus
\begin{eqnarray*}
G(t) &=& \sum_{n \geq 5} G_n t^n  =  
\sum_{n \geq 5} \left( \sum_{i=2}^{n-2} F_{i-2}F_{n-i}\right) t^n \\
&=& \frac{(2+t)t^4}{(1-t-t^2)^2}-2t^4 \\
&=& \frac{(5+2t-4t^2-2t^3)t^5}{(1-t-t^2)^2}.
\end{eqnarray*}

We can now rewrite (\ref{eqnAx^4}) as 
\begin{equation}\label{newAx4rec}
A_n(x)|_{x^4} = 2A_n(x)|_{x^4} + B_{n-1}(x)|_{x^4} + G_n
\end{equation}
for $n \geq 5$. 
Multiplying both sides of (\ref{newAx4rec}) by $t^n$ and 
summing for $n \geq 5$, we obtain 
that 
$$ \mathbb{A}_4(t) -6t^4 = 2(\mathbb{B}_4(t) -2t^4) +t\mathbb{B}_4(t) 
+ G(t).
$$
Solving for $\mathbb{A}_4(t)$ then gives 
$$\mathbb{A}_4(t) = \frac{t^4(6+t-7t^2-t^3+3t^4+t^5)}{(1-t-t^2)^3}.$$
\end{proof}

\subsection{The highest coefficient of $x$ in $A(t,x)$ and $B(t,x)=E(t,x)$}\label{coHighest}

Let $a_n = A_n(x)|_{x^n}$, $b_n = B_n(x)|_{x^n}$, and 
$e_n =E_n(x)|_{x^n}$. Thus, for example, $a_n$ is the 
number of permutations $\pi \in S_n(132)$ such that every element 
of $\pi$ is an occurrence of the 1-box pattern in $\pi$. The identity 
element in $S_n$ and its reverse show that 
$a_n$, $b_n$, and $e_n$ are nonzero for all $n \geq 1$.  Moreover, 
the fact that $B_n(x) = E_n(x)$ for all $n \geq 1$ implies 
$b_n =e_n$ for all $n \geq 1$. 
In this section, we shall compute the 
generating functions 
\begin{equation*}
A(t) =  \sum_{n \geq 0} a_n t^n  \ \mbox{and} \ B(t) =  \sum_{n \geq 1} b_n t^n.
\end{equation*}

\begin{theorem} 
$$A(t)=\frac{1-t+2t^3-\sqrt{1-2t-3t^2+4t^3-4t^4}}{2t^2}$$
and 
$$B(t)=\frac{1+t-2t^2+2t^3-\sqrt{1-2t-3t^2+4t^3-4t^4}}{2(1-t+t^2)}.$$
The initial values for $a_n$ are
$$1,1,2,2,6,10,26,54,134,306,754,\ldots$$
and the initial values for $b_n$ are
$$0,1,1,1,2,3,7,14,33,73,174,\ldots .$$
\end{theorem}

\begin{proof} Our proof of the theorem is very similar to the proofs of Lemma \ref{lemma1} and Theorem~\ref{theorem-enum1}. 

First we claim that for $n \geq 4$, 
\begin{equation}\label{eqBAmax}
b_n=1+ \sum_{k=2}^{n-2}(a_k-b_k).
\end{equation} 
Here  1 corresponds to the decreasing permutation $n(n-1)\cdots 1$, and the sum counts permutations of the form $\pi_{1}\cdots\pi_{n-k-1}\ell(n-k+1)(n-k+2)\cdots n$, where $2\leq k\leq n-2$, $\ell<n-k$ and $\pi_{1}\cdots\pi_{n-k-1}\ell$ is a 132-avoiding permutation on $\{1,\ldots,n-k\}$ with the maximum number of occurrences of the 1-box pattern. There are no other permutations counted by $b_n$. Multiplying both parts of (\ref{eqBAmax}) by $t^n$, summing over all $n\geq 4$, and using the fact that $b_1=b_2 =b_3 =1$, we obtain 
$$B(t)-(t+t^2+t^3)=\frac{t^4}{1-t}+\frac{t^2}{1-t}\left((A(t)-(1+t))-(B(t)-t)\right),$$
from where we get
\begin{equation}\label{BinTermsA}
B(t)=\frac{t-t^2+t^2A(t)}{1-t+t^2}.
\end{equation}

Using the fact that  $S_n(132)=S^{(1)}_n(132)\cup S^{(n)}_n(132)\cup_{2\leq i\leq n-1}S^{(i)}_n(132)$, it is easy to see that for $n \geq 4$,
\begin{equation} \label{anrec}
a_n=b_n+e_n+\sum_{k=2}^{n-2}e_ka_{n-k} = 2b_n +\sum_{k=2}^{n-2}b_ka_{n-k}.
\end{equation}
Multiplying both sides of (\ref{anrec}) by $t^n$ and using 
the facts that $a_0=a_1 =1$ and $a_2 = a_3 =2$, we see that 
$$A(t)-(1+t+2t^2+2t^3)=2(B(t)-(t+t^2+t^3))+(B(t)-t)(A(t)-(1+t)).$$
This leads to
\begin{equation}\label{AinTermsB}
A(t)=\frac{1+t^2+(1-t)B(t)}{1+t-B(t)}.
\end{equation}

Solving the system of equations given by (\ref{BinTermsA}) and (\ref{AinTermsB}) for $A(t)$ and $B(t)$ we get the desired result. \end{proof}

\section{The 1-box pattern on separable permutations}\label{sec4}

In this section we enumerate separable permutations with $m$, $0\leq m\leq 3$, occurrences of the 1-box pattern. 

For two non-empty words, $A$ and $B$, we write $A<B$ to indicate that any element in $A$ is less than each element in $B$. We say that $\pi'=\pi_{i}\pi_{i+1}\cdots\pi_{j}$ is an {\em interval} in a permutation $\pi_1\cdots\pi_n$ if $\pi'$ is a permutation of $\{k,k+1,\ldots,k+j-i\}$ for some $k$, that is, if $\pi'$ consists of consecutive values. 

\fig{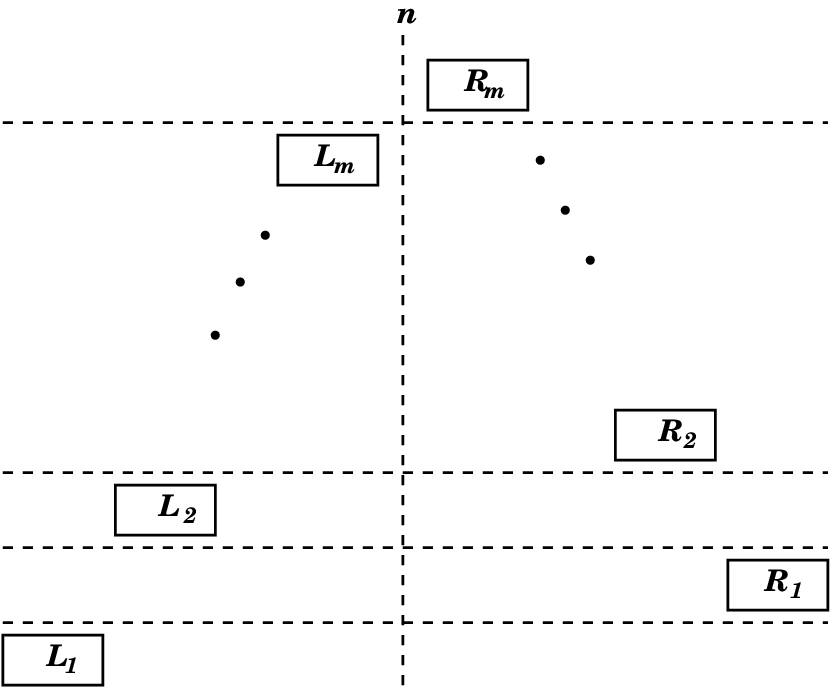}{The structure of a separable permutation.}


A permutation is {\em separable} if it avoids simultaneously the patterns 2413 and 3142. It is known and is not difficult to see that any separable permutation $\pi$ of length $n$ has the following structure (also illustrated in Figure \ref{fig:separable}): 

\begin{equation}\label{structure}\pi=L_1L_2\cdots
L_mnR_mR_{m-1}\cdots R_1\end{equation} 

where
\begin{itemize}
\item for $1\leq i\leq m$, $L_i$ and $R_i$ are non-empty, with possible exception of 
$L_1$ and $R_m$, separable permutations which are intervals in $\pi$, and 
\item $L_1<R_1<L_2<R_2<\cdots <L_m<R_m$. In particular, $L_1$, if it is non-empty,
contains the element 1.
\end{itemize}

For example, if $\pi=215643$ then $L_1=21$, $L_2=5$ $R_1=43$ and
$R_2=\emptyset$. 

The following theorem is similar to the case of 132-avoiding permutations.

\begin{theorem}\label{sepAv} Apart from the empty permutation and the permutation $1$, there are no separable permutations avoiding the $1$-box pattern.\end{theorem}

\begin{proof} Our proof is straightforward by induction on $n$, the length of permutations and is similar to the proof of Theorem \ref{smallest}. Indeed, the base cases for $n\leq 2$ are easy to check. Now assume that $n\geq 3$ and $R_n$ is non-empty (the case when $R_n$ is empty can be considered similarly substituting $R_n$ with $L_n$ in our arguments). If $R_n$ has only one element, $n-1$, then $n$ and $n-1$ give two occurrences of the 1-box pattern; otherwise, $R_n$ contains an occurrence of the pattern by the inductive hypothesis.\end{proof}

By definition of an occurrence of the 1-box pattern, we cannot have any permutations with exactly one occurrence of the 1-box pattern. 

\begin{theorem} The number $c_n$ of separable permutations of length $n$ with exactly two occurrences of the $1$-box pattern is given by $c_0=c_1=0$, $c_2=2$, and for $n\geq 3$, $c_n=2c_{n-1}+c_{n-2}$. The generating function for this sequence is $$\sum_{n\geq 0}c_nt^n=\frac{2t^2}{1-2t-t^2}.$$ The initial values for 
$c_n$s, for $n\geq 0$, are $0,0,2,4,10,24,58,140,338,816,1970,\ldots$, and this is essentially the sequence  $A052542$ in  \cite{oeis}. Apart from the initial $0$s, the sequence of $c_n$s is simply twice the {\em Pell numbers}.\end{theorem} 

\begin{proof} 
%

Suppose that $n\geq 3$ and $\pi$ is a separable permutation in $S_n$ which is counted by $c_n$. 
Thus $\pi$ either contains a consecutive sequence of the form 
$a(a+1)$ or $(a+1)a$. If we remove $a$ from $\pi$ and decrease 
all the elements that are greater than or equal to $a+1$ by one, we will obtain 
a separable permutation $\pi'$ in $S_{n-1}$. By Theorem \ref{sepAv}, we must have at least two occurrences of the pattern in the obtained permutation $\pi'$. 
In fact, it is easy to see  that we will either get two occurrences or three occurrences of the 1-box pattern in $\pi'$. 

By Theorem \ref{separable3occ} below the number of possibilities to get $\pi'$ with three occurrences of the 1-box pattern (necessarily formed by either a consecutive subword of the form $a(a+1)(a+2)$ or by $(a+2)(a+1)a$) is given by $c_{n-2}$. This is indeed the case because we can reverse removing the element in this case by turning $a(a+1)(a+2)$ to $a(a+2)(a+1)(a+3)$ or $(a+2)(a+1)a$ to $(a+3)(a+1)(a+2)a$ and increasing by 1 each element of $\pi$ that is larger than $(a+2)$. On the other hand, the number of possibilities to get $\pi'$ with two occurrences of the 1-box pattern (formed by either a  consecutive elements of the form $a(a+1)$ or by $(a+1)a$) is given by $2c_{n-1}$. Indeed, to reverse removing the element in this case we need either to turn  $a(a+1)$ to either $(a+1)a(a+2)$ or to $a(a+2)(a+1)$, or to turn $(a+1)a$ to either $(a+2)a(a+1)$ or to $(a+1)(a+2)a$. In each of these cases the suggested substitutions create, in an injective way, separable permutations with exactly two occurrences of the 1-box pattern. 

Our considerations above justify the recursion $c_n=2c_{n-1}+c_{n-2}$ (the initial values for it are easy to see). Finally, using the standard technique, it is straightforward to derive the generating function based on the recursion above. \end{proof}

\begin{theorem}\label{separable3occ} For $n\geq 1$, the number of separable permutations of length $n$ with exactly three occurrences of the $1$-box pattern is equal to the number of separable permutations of length $n-1$ with exactly two occurrences of this pattern.\end{theorem}

\begin{proof}
It is easy to see that if a separable permutation has exactly three occurrences of the 1-box pattern, then these occurrences are  necessarily formed by either a consecutive subword of the form $a(a+1)(a+2)$ or by $(a+2)(a+1)a$. In either case, removing the middle element and reducing by 1 all elements that are larger than $(a+1)$, we get a separable permutation with exactly two occurrences of the 1-box pattern. This operation is obviously reversible.    
\end{proof}

Even though we were not deriving formulas for separable permutations with other number of occurrences of the 1-box pattern, we provide initial values for the number of separable permutations with exactly four occurrences of the 1-box pattern (not in \cite{oeis}): 
$$0,0,0,0,8,42,178,664,2288,\ldots, $$
and with the maximum number of occurrences of this pattern on separable permutations  (again, not in \cite{oeis}):
$$0,0,2,2,8,14,54,128,466, \ldots.$$

\end{document}